\documentclass[reqno]{amsart}

\usepackage{amsthm,amssymb,bm,mathabx,fullpage}

\numberwithin{equation}{section} 

\theoremstyle{plain}
\newtheorem{theorem}{Theorem}
\newtheorem{lemma}{Lemma}

\newtheorem{proposition}{Proposition}

\theoremstyle{definition}

\begin{document}

\title{On sums of powers of almost equal primes}
\author{Angel Kumchev}
\address{Department of Mathematics, Towson University, Towson, MD 21252, USA}
\email{akumchev@towson.edu}
\author{Huafeng Liu}
\address{School of Mathematics, Shandong University, Jinan, Shandong 250100, China}
\email{hfliu\_sdu@hotmail.com}
\date{\today}

\begin{abstract}
Let $k \ge 2$ and $s$ be positive integers, and let $n$ be a large positive integer subject to certain local conditions. We prove that if $s \ge k^2+k+1$ and $\theta > 31/40$, then $n$ can be expressed as a sum $p_1^k + \dots + p_s^k$, where $p_1, \dots, p_s$ are primes with $|p_j - (n/s)^{1/k}| \le n^{\theta/k}$. This improves on earlier work by Wei and Wooley \cite{wei2014sums} and by Huang \cite{huang2015exponential} who proved similar theorems when $\theta > 19/24$.
\end{abstract}

\maketitle

\section{Introduction} 
\label{sec:introduction}

The study of additive representations of integers as sums of powers of primes goes back to the work of Hua \cite{hua1938, hua1965}. In particular, Hua proved that when $k$ and $s$ are positive integers with $s>2^k$, every sufficiently large natural number $n$ satisfying certain local solubility conditions can be represented as 
\begin{equation}\label{eq1.1}
  n=p_1^k+\cdots+p_s^k,
\end{equation} 
where $p_1, \dots, p_s$ are prime numbers. (Henceforth, the letter $p$, with or without subscripts, always denotes a prime number.) To describe the local conditions, we let $\tau = \tau(k,p)$ be the largest integer with $p^\tau \mid k$, and then define 
\[ 
K(k) = \prod_{(p-1) \mid k} p^{\gamma(k,p)}, \qquad
\gamma(k,p) = \begin{cases}
  \tau(k,p) + 2 & \text{when } p = 2, \tau > 0, \\
  \tau(k,p) + 1 &\text{otherwise}.
\end{cases} \]
One typically studies \eqref{eq1.1} for $n$ restricted to the congruence class 
\[ \mathcal{H}_{k,s}= \big\{n\in\mathbb{N} : n\equiv s \pmod {K(k)} \big\}. \]

In this paper, we are interested in the additive representations of the form \eqref{eq1.1} with ``almost equal'' primes. Given a large integer $n\in\mathcal{H}_{k,s}$, we ask whether it is possible to solve \eqref{eq1.1} in primes subject to
\begin{equation}\label{eq1.2}
\big| p_j-(n/s)^{1/k} \big|\leq H \qquad (1 \le j \le s), 
\end{equation}
where $H=o(n^{1/k})$. There is a long list of results on sums of five or fewer almost equal squares ($k = 2$, $3 \le s \le 5$), beginning with the work of Liu and Zhan \cite{liuzhan1996} and culminating with the results of Kumchev and Li \cite{kumchev2012li} (see \cite{kumchev2012li} for a detailed history of that problem). In particular, Kumchev and Li showed that when $k=2$ and $s=5$ the problem has solutions with $H = n^{\theta/2}$ for any fixed $\theta > 8/9$. They were also the first to obtain results on sums of more than five almost equal squares, where the extra variables are used to reduce the admissible size of $H$. Let $\theta_{k,s}$ denote the least exponent $\theta$ such that \eqref{eq1.1} and \eqref{eq1.2} with $H = n^{\theta/k}$ can be solved for sufficienly large $n \in \mathcal H_{k,s}$ whenever $\theta > \theta_{k,s}$. Kumchev and Li \cite{kumchev2012li} proved that $\theta_{2,s} \le 19/24$ when $s \ge 17$. The lower bound on $s$ in this theorem was reduced to $s \ge 7$ in a recent paper by Wei and Wooley~\cite{wei2014sums}, in which those authors also established surprisingly strong results for higher values of $k$: they proved that if $s > 2k(k-1)$, one has
\begin{equation}\label{weiwooley}
\theta_{k,s} \le \begin{cases}
  4/5   & \text{if } k=3, \\
  5/6   & \text{if } k \ge 4.
\end{cases}
\end{equation}
Huang \cite{huang2015exponential} further reduced the bound \eqref{weiwooley} to $\theta_{k,s} \le 19/24$ for all $k \ge 3$ and $s > 2k(k-1)$. 

The main goal of the present work is to establish the bound $\theta_{k,s} \le 31/40$ for all $k \ge 2$. We also make use of a recent breakthrough by Bourgain, Demeter and Guth \cite{bourgain2015VMT} to reduce the lower bound on $s$ when $k \ge 4$. Our main result is as follows.

\begin{theorem}\label{main theorem}
Let $k \ge 2$, $s \ge k^2+k+1$, and $\theta > 31/40$. When $n\in \mathcal{H}_{k,s}$ is sufficiently large, equation \eqref{eq1.1} has solutions in primes $p_1, \dots, p_s$ satisfying \eqref{eq1.2} with $H=n^{\theta/k}$.
\end{theorem}

Circle method experts will not be surprised that our methods lead also to improvements on the results established by Wei and Wooley \cite{wei2014sums} and by Huang \cite{huang2015exponential} on solubility for ``almost all'' $n$ and on the number of exceptions for representations by six almost equal squares. Indeed, by adapting the ideas in \cite[\S 9]{wei2014sums}, we obtain the following theorems. 

\begin{theorem}
Let $k \ge 2$, $s > k(k+1)/2$, $\theta > 31/40$, and $N \to \infty$. There is a fixed $\delta > 0$ such that equation \eqref{eq1.1} has solutions in primes $p_1, \dots, p_s$ satisfying \eqref{eq1.2} with $H=n^{\theta/k}$ for all but $O(N^{1-\delta})$ integers $n \le N$ subject to $n \in \mathcal{H}_{k,s}$ (and, when $k = 3$ and $s = 7$, also $9 \nmid n$).  
\end{theorem}

\begin{theorem}
Let $\theta > 31/40$, and $N \to \infty$. Let $E_6(N; H)$ denote the number of integers $n \equiv 6 \pmod {24}$, with $|n - N| \le HN^{1/2}$, such that equation \eqref{eq1.1} with $k = 2$ and $s = 6$ has solutions in primes $p_1, \dots, p_6$ satisfying \eqref{eq1.2}. There is a fixed $\delta > 0$ such that 
\[ E_6(N; N^{\theta/2}) \ll N^{(1-\theta)/2 - \delta}. \]
\end{theorem}


{\bf Notation.}
Throughout the paper, the letter $\epsilon$ denotes a sufficiently small positive real number. Any statement in which $\epsilon$ occurs holds for each positive $\epsilon$, and any implied constant in such a statement is allowed to depend on $\epsilon$. The letter $c$ denotes a constant that depends at most on $k$ and $s$, not necessarily the same in all occurrences. As usual in number theory, $\mu(n)$, $\Lambda(n)$, $\phi(n)$, and $\tau(n)$ denote, respectively, the M\"obius function, von Mangoldt's function, Euler's totient function, and the number of divisors function. We write $e(x) = \exp(2\pi i x)$ and $(a, b) = \gcd(a, b)$, and we use $m\sim M$ as an abbreviation for the condition $ M\leq m < 2M$. If $\chi$ denotes a Dirichlet character, we set $\delta_{\chi}=1$ or $0$ according as $\chi$ is principal or not. The sums $\sum_{\chi \bmod q}$ and $\sum_{\chi \bmod q}^{*}$ denote summations over all the characters modulo $q$ and over the primitive characters modulo $q$, respectively.

\section{Outline of the proof}
\label{sec:outline}

Let $x=(n/s)^{1/k}$, $y=x^{\theta}$, $\mathcal{I}=(x-y, x+y]$, and write
\[ \label{R_{k,s}(n)} R_{k,s}(n) = \sum_{\substack{n=p_1^k+\cdots+p_s^k\\ p_i \in \mathcal{I}}}1. \]
Let $\mathbf 1_{\mathbb P}$ denote the indicator function of the primes, and suppose that we have arithmetic functions $\lambda^{\pm}$ such that, for $m \in \mathcal{I}$,
\begin{equation} \label{lambda upper lower}
\lambda^-(m) \leq \mathbf 1_{\mathbb P}(m) \leq \lambda^{+}(m).
\end{equation}
Then the vector sieve of Br\"udern and Fouvry \cite[Lemma 13]{bruedernfouvry} yields
\begin{equation}\label{vector sieve}
\mathbf 1_{\mathbb P}(m_1)\cdots\mathbf 1_{\mathbb P}(m_5)\geq \sum_{i=1}^5 \lambda^{-}(m_i)\prod_{j \ne i} \lambda^{+}(m_j)- 4\lambda^{+}(m_1)\cdots\lambda^{+}(m_5).
\end{equation}
Thus, by the symmetry of the problem, we have 
\begin{equation}\label{lower bound}
R_{k,s}(n)\geq  5 R_{k,s}(n, \lambda^{-})-4 R_{k,s}(n, \lambda^{+}),
\end{equation}
where 
\[ R_{k,s}(n, \lambda)=\sum_{\substack{n=p_1^k+\cdots+p_{s-5}^k+m_1^k+\cdots+m_5^k\\ p_i, m_j \in \mathcal{I}}}\lambda(m_1)\lambda^{+}(m_2)\cdots\lambda^{+}(m_5). \]

To prove the theorem, we show that one can choose sieve functions $\lambda^\pm$ satisfying \eqref{lambda upper lower} so that the right side of \eqref{lower bound} is positive. Our choice of $\lambda^\pm$ is borrowed from Baker, Harman and Pintz \cite{bakerharman}---namely, $\lambda^-$ and $\lambda^+$ are, respectively, the functions $a_0$ and $a_1$ constructed in \S4 of that paper. In many ways, the functions $\lambda^\pm$ imitate the indicator function $\mathbf 1_{\mathbb P}$ of the primes $p \in \mathcal I$. We will discuss the similarities in detail later (see \S\ref{sec: sieve} below) and will focus here on their most crucial property: 
\begin{itemize}
\item [(A0)] Let $A,B > 0$ be fixed (possibly large) numbers and let $x \to \infty$. If $\chi$ is a Dirichlet character modulo $q \le (\log x)^B$ and $x^{11/20+\epsilon} \le y \le x\exp\big( -(\log x)^{1/3} \big)$, then one has 
\begin{equation}\label{def kappa} 
\sum_{|m - x| \le y} \lambda^\pm(m)\chi(m) = \frac {2y}{\phi(q)\log x} \big( \delta_\chi\kappa_\pm + O\big( (\log x)^{-A} \big) \big), 
\end{equation}
where $\kappa_{\pm}$ are absolute constants satisfying
\begin{equation}\label{bounds kappa} 
\kappa_- > 0.99, \qquad \kappa_+ < 1.01. 
\end{equation}
\end{itemize}

We now sketch the application of the circle method to $R_{k,s}(n, \lambda)$. Let $\delta > 0$ be a fixed number, to be chosen later sufficiently small in terms of $k,s$ and $\theta$, and set 
\begin{equation}\label{P Q}
P=y^{\delta},\quad Q=x^{k-2}y^2P^{-1},\quad L=\log x.
\end{equation}
We write  
\[ \mathfrak{M}(q,a)= \big \{\alpha\in\mathbb{R}:|q\alpha-a|\leq Q^{-1} \big\}, \]
and define the sets of major and minor arcs by
\begin{equation}\label{major and minor}
\mathfrak{M}=\bigcup_{\substack{1\leq a \leq q \leq P\\ (a,q)=1}}\mathfrak{M}(q,a) \quad \text{and} \quad \mathfrak{m}= \big[ Q^{-1}, 1+Q^{-1} \big] \setminus \mathfrak{M},
\end{equation}
respectively. Further, for any Lebesgue measurable set $\mathfrak{B}$, we write
\[ R_{k,s}(n, \lambda; \mathfrak{B}) = \int_{\mathfrak B}f(\alpha, \mathbf 1_{\mathbb P})^{s-5} f(\alpha, \lambda) f(\alpha, \lambda^{+})^{4} e(-n\alpha) \, d\alpha, \]
where 
\begin{equation} \label{f}
f(\alpha,\lambda) = \sum_{m\in\mathcal{I}} \lambda(m)e(m^k\alpha).
\end{equation} 
By orthogonality and \eqref{major and minor}, we have
\begin{align}\label{integral}
R_{k,s}(n, \lambda) = R_{k,s}(n, \lambda; \mathfrak{M}) + R_{k,s}(n, \lambda; \mathfrak{m}).
\end{align}
In \S\ref{sec: minor arcs}, we show that when $s \ge k^2+k+1$, $\delta < 1/(16k)$, and $\theta \ge 31/40$, one has
\begin{equation} \label{minor arc bound}
R_{k,s}(n, \lambda; \mathfrak{m}) \ll y^{s-1-\delta/(3k)}x^{1-k}.
\end{equation}
Then, in \S\ref{sec: major arcs}, we show that when $\delta \le 2(\theta - 31/40)$, one has 
\begin{equation} \label{major arc bound}
R_{k,s}(n, \lambda^\pm; \mathfrak{M}) = \mathfrak C(n) y^{s-1}x^{1-k}L^{-s}\big( \kappa_\pm\kappa_+^4 + O(L^{-1}) \big),
\end{equation}
where $1\ll \mathfrak C(n) \ll 1$ for sufficiently large $n \in \mathcal H_{k,s}$, and $\kappa_\pm$ are the constants from \eqref{def kappa}. Theorem~\ref{main theorem} follows from \eqref{lower bound}, \eqref{bounds kappa}, and \eqref{integral}--\eqref{major arc bound}. \qed

\section{The sieve weights}
\label{sec: sieve}

As we said before, we use sieve weights $\lambda^\pm$ constructed by Baker, Harman and Pintz \cite{bakerharman} to have properties \eqref{lambda upper lower} and (A0) above. We remark that (A0) is a short-interval version of the Siegel--Walfisz theorem: when the functions $\lambda^\pm$ are replaced by $\mathbf 1_{\mathbb P}$, the asymptotic formula \eqref{def kappa} with $\kappa = 1$ and $y \ge x^{7/12 + \epsilon}$ is a well-known extension of a celebrated result of Huxley~\cite{huxley1972difference}. In this section, we record some additional properties of the weights $\lambda^\pm$ that we will need later in the paper:
\begin{itemize}
\item [(A1)] The functions $\lambda^\pm(m)$ vanish if $m$ has a prime divisor $p < x^{1/10}$.
\item [(A2)] Let $\mathbb S = \{ p^j : p \in \mathbb P, j \ge 2 \}$. When $m \sim 2x/3$, one can express $\lambda^\pm(m)$ as a linear combination of a bounded function supported on $\mathbb S$ and of $O(L^c)$ triple convolutions of the form
\[ \sum_{\substack{m=uvw\\ u\sim U, \ v\sim V}}\xi_{u}\eta_{v}\zeta_{w}, \]
where $|\xi_u|\leq \tau(u)^c$, $|\eta_v|\leq \tau(v)^c$, $\max(U,V)\ll x^{11/20}$, and either $\zeta_w=1$ for all $w$, or $|\zeta_w|\leq \tau(w)^c$ and $UV\gg x^{27/35}$.
\item [(A3)] Let $A, B, \epsilon > 0$ be fixed, let $\chi$ be a Dirichlet character modulo $q \le L^B$, and put $T_0 = \exp( L^{1/3} )$ and $T_1 = x^{9/20 - \epsilon}$. Then
\[ \int_{T_0}^{T_1} \bigg| \sum_{m \sim 2x/3} \lambda^\pm(m)\chi(m)m^{-1/2-it} \bigg| \, dt \ll x^{1/2}L^{-A}. \]
\end{itemize}

Of the three properties above, (A3) is the easiest to justify, since it is a part of the proof of (A0) in \cite{bakerharman}. Indeed, the method of Baker, Harman and Pintz reduces \eqref{def kappa} to the classical Siegel--Walfisz theorem by decomposing $\lambda^\pm$ into a linear combination of $O(L^c)$ arithmetic functions for which (A3) holds and then applying \cite[Lemma 11]{bakerharman} to each of them. In order to justify that the functions $\lambda^\pm$ have also properties (A1) and (A2), we need to provide some details on their construction. 

The core idea behind the construction of $\lambda^\pm$ is explained in \cite[pages 32--33, 41--42]{bakerharman}. It amounts to setting
\begin{equation}\label{def: lambda}
\lambda^\pm(m) = \mathbf 1_{\mathbb P}(m) \pm \sum_{j=1}^{J^\pm} \lambda_j^\pm(m)
\end{equation}
where $J^\pm = O(1)$ and the arithmetic functions $\lambda_j^\pm$ have the form
\[ \lambda_j^\pm(m) = \sum_{m = u_1 \cdots u_{d+1}} \xi(u_1, \dots, u_{d+1}) \qquad (4 \le d \le 7), \]
with $\xi(u_1, \dots, u_{d+1}) = 1$ or $0$. The latter functions impose various restrictions on the sizes and arithmetic properties of $u_1, \dots, u_{d+1}$ that amount to restricting the support of $\lambda_j^\pm$ to integers $m$ with very specific (undesirable) factorizations.  Moreover:
\begin{itemize}
\item [(i)] Only the cases $d=4$ and $d=6$ occur in the construction of $\lambda^-$, whereas only $d=5$ and $d=7$ occur in the construction of $\lambda^+$.
\item [(ii)] $\xi(u_1, \dots, u_{d+1}) = 0$ if any of $u_1, \dots, u_{d+1}$ has a prime divisor $<x^{1/10}$. Note that property (A1) is an immediate consequence of this observation.
\item [(iii)] When $d = 5$, $\lambda_j^+$ is supported on integers $m$ that have a divisor $u$ in the range $x^{0.46} \le u \le x^{1/2}$: see \cite[p. 42]{bakerharman}.
\item [(iv)] When $d = 4$, $\lambda_j^-$ is supported on integers $m = n_1n_2n_3$, where $n_i = x^{\alpha_i}$ with $\bm\alpha = (\alpha_1, \alpha_2)$ lying in one of regions $\Gamma$, $\Delta_2$, $\Delta_3$, or $\Delta_4$ in \cite[Diagram 1 on p. 33]{bakerharman}.
\end{itemize}

We now turn to property (A2). We note that when $\lambda_j^\pm$ is supported on integers $m = uv$, with $x^{9/20} \le u \le x^{11/20}$, it has property (A2). Thus, by (iii) above, property (A2) holds for all terms $\lambda_j^+$ with $d = 5$. Moreover, the same is true for $\lambda_j^-$ with $d = 4$ and $\bm\alpha$ in one of the regions $\Delta_3$ or $\Delta_4$: we have $0.46 \le \alpha_1 \le 0.5$ when $\bm\alpha \in \Delta_4$, and $0.46 \le \alpha_1+\alpha_2 \le 0.54$ when $\bm\alpha \in \Delta_3$.

We next consider the case $d \ge 6$ and suppose that the variables $u_i$ have been labelled so that $u_1 \ge u_2 \ge \cdots \ge u_{d+1}$. When $\lambda_j^\pm$ is supported on integers $m = u_1 \cdots u_{d+1}$ with $u_4 \cdots u_{d+1} \ge x^{11/20}$, we have 
\[ u_1u_2u_3 \ll x^{9/20} \quad \text{and} \quad u_4 \le \sqrt[3]{u_1u_2u_3} \ll x^{3/20}. \]  
Since $u_5 \cdots u_{d+1} \ll x^{1/2}$, we can then verify that $\lambda_j^\pm$ has property (A2) by grouping the variables $u_1, \dots, u_{d+1}$ into $u = u_1u_2u_3$, $v = u_5 \cdots u_{d+1}$, and $w = u_4$. On the other hand, when $\lambda_j^\pm$ is supported on integers $m = u_1 \cdots u_{d+1}$ with $u_4 \cdots u_{d+1} \le x^{11/20}$, we note that
\[ u_1u_2 \ll x^{1/2} \quad \text{and} \quad u_3 \le \sqrt[3]{u_1u_2u_3} \ll x^{1/5}. \]  
Thus, we can verify that $\lambda_j^\pm$ has property (A2) by grouping the variables $u_1, \dots, u_{d+1}$ into $u = u_1u_2$, $v = u_4 \cdots u_{d+1}$, and $w = u_3$.  
  
The functions $\lambda_j^-$ with $d = 4$ and $\bm\alpha \in \Delta_2$ are supported on integers $m = u_1 \cdots u_5$, where 
\begin{equation}\label{eq3.2} 
x^{1/10} \le u_4 \le u_3 \le u_2 \le u_1, \quad \text{and} \quad x^{0.32} \le u_1u_2 \le x^{0.36}. 
\end{equation}
(These functions arise by ``decomposing twice the variable $n_3$'' in \cite[(4.24)]{bakerharman}, so we have $u_1u_2 = x^{\alpha_1+\alpha_2}$.) Since the inequalities \eqref{eq3.2} imply that
\[ x^{1/10} \le u_4 \le u_3 \le x^{0.18}, \quad u_1u_2u_3 \le x^{0.54}, \quad u_5 \ll x^{0.48}, \]
we can verify that $\lambda_j^-$ has property (A2) by grouping the variables $u_1, \dots, u_5$ into $u = u_1u_2u_3$, $v = u_5$, and $w = u_4$. Similarly, the functions $\lambda_j^-$ with $d = 4$ and $\bm\alpha \in \Gamma$ are supported on integers $m = u_1 \cdots u_5$, where 
\[ x^{0.32} \ll u_1u_2, u_3u_4 \ll x^{0.36}, \quad \text{and} \quad u_5 \le x^{1/3}. \]
(In this case, we have $u_1u_2 = x^{\alpha_1}$ and $u_5 = x^{\alpha_2}$.) If we assume that the variables are labelled so that $u_1 \le u_2$ and $u_3 \le u_4$, we have
\[ u_2u_4 \le x^{0.72}/(u_1u_3) \le x^{0.52}, \quad u_1u_5 \le x^{0.18}x^{1/3} < x^{0.52}, \quad u_3 \le x^{0.18}. \]
Hence, we can once again verify that $\lambda_j^-$ has property (A2) by grouping the variables $u_1, \dots, u_5$ into $u = u_2u_4$, $v = u_1u_5$, and $w = u_3$. 

We have shown that each term $\lambda_j^\pm$ on the right side of \eqref{def: lambda} satisfies (A2). It remains to show that so does the indicator function $\mathbf 1_{\mathbb P}$. The proof of \cite[Theorem 1.1]{choi2006kumchev} uses Heath-Brown's identity to establish (A2) for von Mangoldt's function. In the case of $\mathbf 1_{\mathbb P}$, we can use a variant of that argument based on Linnik's identity instead of Heath-Brown's.

\section{The minor arcs}
\label{sec: minor arcs}
 
In this section, we establish inequality \eqref{minor arc bound}. Our main tools are Propositions~\ref{proposition: mean value} and~\ref{proposition: exp sum} below.

\begin{proposition}\label{proposition: mean value}
Suppose that $k \ge 2$, $s \geq k^2+k$, and $y \geq x^{1/2}$. Then for any bounded arithmetic function $\lambda$, one has
\begin{equation}
I_s(\lambda) := \int_0^1 |f(\alpha, \lambda)|^s \, d\alpha  \ll  y^{s-1} x^{1-k+\epsilon}.
\end{equation}
\end{proposition}

\begin{proposition}\label{proposition: exp sum}
Let $k \ge 2$, $0 < \delta < 1/(16k)$, and $y \ge x^{31/40}$, and suppose that $\alpha \in \mathfrak m$. Then
\[ f(\alpha, \mathbf 1_{\mathbb P}) \ll y^{1 - \delta/(2k) + \epsilon}. \]
\end{proposition}

It is straightforward to deduce \eqref{minor arc bound} from these propositions. First, we remark that the functions $\lambda^\pm$ are bounded by construction---they are linear combinations of a bounded number of indicator functions. Thus, we may apply Proposition \ref{proposition: mean value} to $\lambda = \lambda^\pm$. By H\"older's inequality,
\[ |R_{k,s}(n, \lambda; \mathfrak m)| \le \Big( \sup_{\alpha \in \mathfrak m} |f(\alpha, \mathbf 1_{\mathbb P})| \Big) I_{s-1}(\lambda)^u I_{s-1}(\lambda^+)^{4u} I_{s-1}(\mathbf 1_{\mathbb P})^{1-5u}, \]
where $u = (s-1)^{-1}$. Thus, when $s \ge k^2+k+1$, we may use Propositions \ref{proposition: mean value} and \ref{proposition: exp sum} to get
\[ R_{k,s}(n, \lambda; \mathfrak m) \ll y^{1-\delta/(2k)+\epsilon} y^{s-2}x^{1-k+\epsilon} \ll y^{s-1-\delta/(3k)} x^{1-k}, \]
provided that $\delta$ and $y$ satisfy the hypotheses of Proposition \ref{proposition: exp sum} and $\epsilon$ is chosen sufficiently small; this verifies \eqref{minor arc bound}. In the remainder of this section, we prove the propositions.

\subsection{Proof of Proposition \ref{proposition: mean value}}

This is a variant of \cite[Proposition 2.2]{wei2014sums}, which we have extended in two ways. First, we have included the arbitrary coefficients $\lambda$. This is straightforward, due to the ``maximal inequality''
\begin{equation}\label{max inequality}
\int_0^1 |f(\alpha, \lambda)|^{s} \, d\alpha \ll y^{s-k^2-k}\int_0^1 |f(\alpha, \mathbf 1)|^{k^2+k} \, d\alpha,
\end{equation}
where $\mathbf 1$ is the constant function $\mathbf 1(n) = 1$ (compare this to \cite[p. 1136]{wei2014sums}). Like Wei and Wooley, we estimate the right side of \eqref{max inequality} by means of \cite[Theorem 3]{daemen2010asymptotic} and standard bounds for Vinogradov's mean-value integral. In particular, the recent work of Bourgain, Demeter and Guth \cite{bourgain2015VMT} allows us to reduce the lower bound on $s$ to the one stated above. \qed

\subsection{Proof of Proposition \ref{proposition: exp sum}}

Although it looks somewhat different, Proposition \ref{proposition: exp sum} is merely a slight variation of the main theorem of Huang \cite{huang2015exponential}, and our proof follows closely Huang's. We first obtain variants of some technical estimates from~\cite{huang2015exponential} by making some slight changes to Huang's arguments. 

\begin{lemma}\label{Weyl bound}
Let $k \geq 2$ be an integer and $\rho$ be real, with $0 < \rho \le t_k^{-1}$, where
\[ t_k = \begin{cases} 2 &\text{if } k = 2, \\ k^2 - k + 1 &\text{if } k \ge 3. \end{cases} \]
Suppose also that $y = x^{\theta}$, where 
\[ \frac 1{2 - t_k\rho} \leq \theta \leq 1. \] 
Then either
\[ \sum_{x < m \le x+y} e(m^k\alpha) \ll y^{1-\rho+\epsilon}, \]
or there exist integers $a,q$ such that
\[ 1 \le q \le y^{k\rho}, \quad (a,q) = 1, \quad |q\alpha - a| \le x^{1-k}y^{k\rho - 1}, \]
and
\[ \sum_{x < m \le x+y} e(m^k\alpha) \ll y^{1-\rho+\epsilon} + \frac {y}{(q + yx^{k-1}|q\alpha - a|)^{1/k}}. \]
\end{lemma}

\begin{proof}
When $k \ge 3$, we follow the argument of Huang \cite[Lemma 1]{huang2015exponential} with $\gamma = \rho^{-1} (t_k-1)^{-1}$. Within that argument, we apply the latest version of Vinogradov's mean-value theorem due to Bourgain, Demeter and Guth \cite{bourgain2015VMT} in place of the earlier version by Wooley \cite{wooley2016cubic} used by Huang. When $k = 2$, we follow the same argument with $\gamma = (2\rho)^{-1}$ but observe that in this case the bound at the top of \cite[p. 512]{huang2015exponential} can be improved to 
\[ \Delta \ll q^{1/2 + \epsilon}(1 + x^2(qQ_0)^{-1})^{1/2} \ll P_0^{1/2 + \epsilon}xy^{-1}. \]
This slight improvement is possible, because in the quadratic case, Daemen's proof of \cite[(3.5)]{daemen2010asymptotic} does not require the iterative process in \cite[p. 78]{daemen2010asymptotic}. Thus, we need not incur a loss of a factor of $q^{-1/2}$ in the above bound which the iterative method causes when $k\ge 3$. 
\end{proof}

\begin{lemma}[Type II sum]\label{type 2}
Let $k \geq 2$ be an integer, let $\rho$ be real, with $0 < \rho \le \min\big( (4t_k)^{-1} , \frac 1{20} \big)$, and suppose that $y = x^{\theta}$, where
\begin{equation}\label{thetarange} 
\frac 3{4 - 4t_k\rho} \leq \theta \leq 1.
\end{equation} 
Suppose also that $\alpha \in \mathfrak m$ and that the coefficients $\xi_u, \eta_v$ satisfy $\xi_u \ll \tau(u)^c$ and $\eta_v \ll \tau(v)^c$. Then 
\begin{equation*}
\sum_{u\sim U}\sum_{uv \in \mathcal I} \xi_u\eta_v e(u^k v^k \alpha) \ll y^{1-\rho +\epsilon} + y^{1+\epsilon}P^{-1/(2k)},
\end{equation*}
provided that 
\begin{equation}\label{type2range} 
xy^{-1+2\rho} \ll U \ll y^{1-2\rho}.
\end{equation} 
\end{lemma}

\begin{proof}
This is a version of \cite[Proposition 2]{huang2015exponential} that applies Lemma \ref{Weyl bound} above in place of   \cite[Lemma~1]{huang2015exponential}. We have also altered slightly the choice of $\nu$ in Huang's argument by choosing it so that $Y^\nu = y^{2\rho}L^{-1}$ as opposed to $Y^\nu = x^{2\rho}L^{-1}$ (see \cite[p. 515]{huang2015exponential}).
\end{proof}

\begin{lemma}[Type I sum]\label{type 1}
Let $k \geq 2$ be an integer, let $\rho$ be real, with $0 < \rho \le \min\big( (4t_k)^{-1} , \frac 1{20} \big)$, and suppose that $y = x^{\theta}$, with $\theta$ in the range \eqref{thetarange}. Suppose also that $\alpha \in \mathfrak m$ and that the coefficients $\xi_u$ satisfy $\xi_u \ll \tau(u)^c$. Then 
\begin{equation*}
\sum_{u\sim U}\sum_{uv \in \mathcal I} \xi_u e(u^k v^k \alpha) \ll y^{1-\rho +\epsilon} + y^{1+\epsilon}P^{-1/(2k)},
\end{equation*}
provided that 
\begin{equation}\label{type1range} 
U \ll y^{1-2\rho}.
\end{equation} 
\end{lemma}

\begin{proof}
This is a version of \cite[Proposition 1]{huang2015exponential}. Following the proof of that result, with our Lemma \ref{Weyl bound} in place of \cite[Lemma 1]{huang2015exponential} and with $\nu$ chosen so that $Y^\nu = y^{\rho}L^{-1}$, one obtains the above bound when 
\[  U \ll x^{-1}y^{2 - t_k\rho}, \qquad U^{2k} \ll x^{k-1}y^{1-2k\rho}.  \]
On the other hand, when either of these inequalities fails, one has $U \gg xy^{-1+2\rho}$ and the result follows from Lemma \ref{type 2}.
\end{proof}

\begin{proof}[Proof of Proposition \ref{proposition: exp sum}]
It suffices to bound $f(\alpha, \Lambda)$, where $\Lambda$ is von Mangoldt's function. Let $\rho = (31t_k)^{-1}$ and $X = xy^{-1+2\rho}$. We note that this choice of $\rho$ ensures that \eqref{thetarange} holds for all $\theta \ge 31/40$ and that $X \le x^{9/40 + (31\rho)/20} \le x^{1/4}$. We may thus apply Vaughan's identity for $\Lambda$ (see \cite[p. 28]{vaughan book}) to decompose $f(\alpha, \Lambda)$ into $O(L)$ type I sums with $U \le X^2$ and $O(L)$ type II sums with $X \le U \le xX^{-1}$. By the choice of $X$ and $\rho$, Lemma \ref{type 2} can be applied to the arising type II sums. Moreover, since $X^2 \le xX^{-1} = y^{1-2\rho}$, Lemma \ref{type 1} can be applied to the type I sums. We conclude that when $\alpha \in \mathfrak m$, one has
\[ f(\alpha, \Lambda) \ll y^{1-\rho+\epsilon} + y^{1-\delta/(2k)+\epsilon}. \]
Since the hypothesis $\delta < 1/(16k)$ ensures that $\delta/(2k) < \rho$, this completes the proof.
\end{proof}

\section{The major arcs}
\label{sec: major arcs}

In this section, we establish \eqref{major arc bound}. First, we need to introduce some notation. We write 
\[ S(q,a)=\sum_{\substack{1\leq h\leq q\\ (h,q)=1}}e(ah^k/q), \quad v(\beta; s) = \int_{\mathcal I} u^{s-1}e(u^k\beta) \, du, \]
and define the singular series $\mathfrak S(n)$ and the singular integral $\mathfrak{I}(n)$ by
\[ \mathfrak{S}(n)=\sum_{q=1}^{\infty} \phi(q)^{-s} \sum_{\substack{1\leq a \leq q\\(a,q)=1}} S(q,a)^s e(-an/q), \quad \mathfrak{I}(n)=\int_{\mathbb R} v(\beta;1)^se(-n\beta) \, d\beta. \]
If $\lambda$ denotes one of the functions $\lambda^\pm$ and $\kappa$ the respective constant $\kappa_\pm$, we define a function $f^*(\alpha,\lambda)$ on the major arcs $\mathfrak{M}$ by setting
\[ f^*(\alpha, \lambda)= \kappa\phi(q)^{-1}S(q,a)v(\beta; 1)L^{-1} \quad \text{if } \alpha \in \mathfrak{M}(q,a). \]
This is the ``major arc approximation'' to $f(\alpha, \lambda)$. We also define a major arc approximation to $f(\alpha, \mathbf 1_{\mathbb P})$ by
\[ f^*(\alpha)= \phi(q)^{-1}S(q,a)v(\beta; 1)L^{-1} \quad \text{if } \alpha \in \mathfrak{M}(q,a). \]
Finally, we adopt the convention that for any arithmetic function $\lambda$, there is an associated Dirichlet polynomial $F(s,\lambda)$, given by
\[ F(s, \lambda) = \sum_{m \sim 2x/3} \lambda(m)m^{-s}. \] 

\subsection{Some technical estimates}

\begin{lemma}\label{W lemma}
Let $x^{11/20} \le y \le x$ and suppose that $P, Q$ satisfy 
\[ PQ \le yx^{k-1}, \quad Q \ge x^{k-9/20}. \]
Suppose also that $g$ is a positive integer, $\nu > 1$, and $\lambda$ is a bounded arithmetic function satisfying hypothesis \emph{(A2)} above. Then
\begin{equation} \label{g r}
\sum_{r\leq P} [g,r]^{-\nu} \sideset{}{^*}\sum_{\chi \bmod r} \bigg( \int_{-1/(rQ)}^{1/(rQ)}|f(\beta, \lambda\chi)|^2 d\beta \bigg)^{1/2} \ll g^{-\nu+\epsilon} y^{1/2}x^{(1-k)/2}L^c.
\end{equation}
\end{lemma}

\begin{proof}
When $k = 2$ and $\nu=1-\epsilon$, this is \cite[Lemma 4.5]{kumchev2012li}. The proof for general $k \ge 2$ and $\nu \ge 1$ uses the same argument with some obvious changes: e.g., $T_1 = \Delta x^k$ and $H \ll \Delta^{-1}x^{1-k}$ in place of the respective statements in  \cite[p. 618]{kumchev2012li}. 
\end{proof}

\begin{lemma}\label{lemma Perron}
Let $x$ be a large integer, and suppose that $y, b, T$ are reals with: $y = o(x)$, $\| y \| = 1/2$, $0 < b \le 1$, and $1 \le T \le x^{1/2}$. Suppose also that $\lambda$ is a bounded arithmetic function. Then 
\[
  f(\beta, \lambda) = \frac 1{2\pi i} \int_{b-iT}^{b+iT} F(s,\lambda) v(\beta; s) \, ds + O\big( (1+yx^{k-1}|\beta|)xLT^{-1} \big).
\]
\end{lemma}

\begin{proof}
For any $u \in \mathcal I$ with $\| u \| = 1/2$, Perron's formula (see \cite[Corollary 5.3]{montgomery book}) gives
\begin{equation}\label{lambda chi mean value}
\sum_{x - y < m \le u} \lambda(m) = \frac{1}{2\pi i}
\int_{b-iT}^{b+iT} F(s, \lambda) \frac{u^s-(x-y)^s}{s} \, ds + O(xLT^{-1}).
\end{equation}
If we change $u$ in \eqref{lambda chi mean value} to $u_1$, where $|u_1 - u| \le 1/2$, the left side will change by $O(1)$ and the integral on the right side will change by $O(T)$. Hence, the integral representation \eqref{lambda chi mean value} can be extended to all $u \in \mathcal I$. The conclusion of the lemma then follows by partial summation.
\end{proof}

\begin{lemma} \label{V lemma}
Under the assumptions of Lemma \ref{W lemma}, we have 
\begin{equation}\label{eq:lemV.1}
\sum_{r \leq P} [g, r]^{-\nu} \sideset{}{^*}\sum_{\chi \bmod r}
 \max_{|\beta|\leq 1/(rQ)}|f(\beta, \lambda\chi)|\ll g^{-\nu+\epsilon}yL^c.
\end{equation}
Furthermore, for any given $A>0$, there is a $B=B(A, \nu)>0$ such that 
\begin{equation}\label{eq:lemV.2}
\sum_{L^B \leq r\leq P}r^{-\nu} \sideset{}{^*}\sum_{\chi \bmod r} \max_{|\beta|\leq 1/(rQ)}|f(\beta, \lambda\chi)| \ll yL^{-A}.
\end{equation}
\end{lemma}

\begin{proof}
Let $1 \le R_0 \le P$. By a simple splitting argument, 
\begin{equation}\label{eq:lem6a}
\sum_{R_0 \le r \leq P} [g, r]^{-\nu} \sideset{}{^*}\sum_{\chi \bmod r}
 \max_{|\beta|\leq 1/(rQ)}|f(\beta, \lambda\chi)| \ll (gR)^{-\nu}L\sum_{\substack{d \mid g\\d \le 2R}} d^{\nu} S(R,d),
\end{equation}
where $R_0 \le R \le P$ and 
\[
  S(R,d) = \sum_{\substack{r \sim R\\ d \mid r}} \; \sideset{}{^*}\sum_{\chi \bmod r}
 \max_{|\beta|\leq 1/(RQ)}|f(\beta, \lambda\chi)|.
\]

We now estimate $S(R,d)$. The contribution to $S(R,d)$ from any powers of primes in the support of $\lambda$ can be bounded trivially as $O( yx^{-1/2}(R^2/d) )$. Under the assumptions of the lemma, we have $P \le yx^{-11/20}$, so this contribution can be absorbed into the term $y(R/d)L$ on the right side of \eqref{eq:lem6d} below. Thus, we may assume that $\lambda$ is merely the linear combination of triple convolutions of the kind described in (A2). We may also assume that $x \in \mathbb Z$ and $\|y\| = 1/2$. 

Let $0 < b \le 1$, $|\beta| \le (RQ)^{-1}$, $T_1 = 3k\pi x^kQ^{-1}$, and $T_0 = T_1/R$. Then, by Lemma~\ref{lemma Perron} with $T = T_1$, 
\begin{equation}\label{eq:lem6b}
f(\beta, \lambda\chi) = \frac 1{2\pi i} \int_{b-iT_1}^{b+iT_1} F(s,\lambda\chi) v(\beta; s) \, ds + O(yR^{-1}L).
\end{equation}
Letting $b \downarrow 0$ in \eqref{eq:lem6b}, we obtain
\begin{equation}\label{eq:lem6c}
  f(\beta, \lambda\chi) = \frac 1{2\pi} \int_{-T_1}^{T_1} F(it, \lambda\chi) v(\beta; it) \, dt + O(yR^{-1}L).
\end{equation}
When $|\beta| \le (RQ)^{-1}$ and $|t| \ge T_0$, we have
\[
  v(\beta; it) \ll |t|^{-1},
\]
by the first-derivative test for exponential integrals (see \cite[Lemma 4.5]{titchmarsh book}). Combining this bound with \eqref{eq:lem6c} and the trivial estimate $|v(\beta; it)| \ll yx^{-1}$, we find that 
\[
  f(\beta, \lambda\chi) \ll yx^{-1} \int_{-T_0}^{T_0} |F(it, \lambda\chi)| \, dt + \int_{T_0 \le |t| \le T_1} |F(it, \lambda\chi)| \, \frac {dt}{|t|} + yR^{-1}L.
\]
Summing this inequality over $r$ and $\chi$ and then splitting the range of $t$ in the second integral into dyadic intervals, we deduce that
\begin{equation}\label{eq:lem6d}
  S(R,d) \ll yx^{-1} S_1(R,d; T_0) + \sum_{2^j \le R} (2^jT_0)^{-1}S_1(R,d;2^jT_0) + y(R/d)L,
\end{equation}
where 
\[
  S_1(R,d;T) = \sum_{\substack{r \sim R\\ d \mid r}} \; \sideset{}{^*}\sum_{\chi \bmod r} \int_{-T}^T |F(it, \lambda\chi)| \, dt.
\]
Since $\lambda$ is assumed to be a linear combination of convolutions of the type in (A2), we may apply \cite[Theorem 2.1]{choi2006kumchev} to obtain the bound
\[
  S_1(R,d;T) \ll \big( x + (R^2T/d)x^{11/20} \big)L^c.
\]
Combining this bound, \eqref{eq:lem6a} and \eqref{eq:lem6d}, we conclude that the left side of \eqref{eq:lemV.1} is
\[
  \ll g^{-\nu+\epsilon}y \big( 1 + x^{k-9/20}Q^{-1} + x^{1-k}y^{-1}PQ + Px^{11/20}y^{-1} \big)L^c.
\]
This establishes the first claim of the lemma. 

When $g = 1$, the above argument yields the bound
\[
  \ll yR_0^{1-\nu} \big( 1 + x^{k-9/20}Q^{-1} + x^{1-k}y^{-1}PQ + Px^{11/20}y^{-1} \big)L^c
\]
for the left side of  \eqref{eq:lemV.2}. When $R_0 = L^B$ for a sufficiently large $B > 0$, this establishes the second claim of the lemma.  
\end{proof}

\begin{lemma} \label{main lemma}
Let $x^{11/20+2\epsilon} \le y \le x^{1-\epsilon}$ and suppose that $P,Q$ satisfy
\begin{equation}\label{eq5.9} 
PQ \le yx^{k-1}, \quad Q \ge x^{k-9/20+2\epsilon}. 
\end{equation}
Suppose also that $\nu > 1$ and $\lambda$ is a bounded arithmetic function that satisfies hypotheses \emph{(A0)}, \emph{(A2)} and \emph{(A3)} above. Then, for any given $A>0$,
\begin{equation}\label{eq:lemV.3}
\sum_{r \leq P}r^{-\nu} \sideset{}{^*}\sum_{\chi \bmod r} \max_{|\beta|\leq 1/(rQ)} |f(\beta, \lambda\chi) - \rho_\chi v(\beta;1)| \ll yL^{-A},
\end{equation}
where $\rho_\chi = \delta_\chi\kappa L^{-1}$, $\kappa$ being the constant in hypothesis \emph{(A0)} for $\lambda$.
\end{lemma}

\begin{proof}
By the second part of Lemma \ref{V lemma}, it suffices to show that
\begin{equation}\label{eq:lem7a} 
\max_{|\beta| \le 1/Q} |f(\beta, \lambda\chi) - \rho_\chi v(\beta;1) | \ll yL^{-B-A} 
\end{equation}
for all primitive characters $\chi$ with moduli $r \le L^B$, where $B = B(A, \nu)$ is the number that appears in \eqref{eq:lemV.2}. 
Let $\chi$ be such a character and suppose that $|\beta| \le Q^{-1}$. By Lemma~\ref{lemma Perron} with $b = 1/2$ and $T = T_1 = x^{9/20-\epsilon}$, 
\begin{equation}\label{eq:lem7b}
f(\beta, \lambda\chi) = \frac 1{2\pi i} \int_{1/2-iT_1}^{1/2+iT_1} F(s,\lambda\chi) v(\beta; s) \, ds + O\big( yx^{-\epsilon/2} +  yx^{k-9/20+\epsilon}Q^{-1}L \big).
\end{equation}
Since $v(\beta; 1/2+it) \ll yx^{-1/2}$, we deduce from \eqref{eq:lem7b} and hypothesis (A3) that
\[ f(\beta, \lambda\chi) = \frac{1}{2\pi i} \int_{1/2-iT_0}^{1/2+iT_0} F(s, \lambda\chi)v(\beta; s) \, ds + O(yL^{-B-A}), \]
where $T_0 = \exp(L^{1/3})$. Note that when $\mathrm{Re}(s) = 1/2$,  
\[ v(\beta; s) - x^{s-1}v(\beta; 1) \ll (|s|+1)y^2x^{-3/2}. \]
Hence, 
\begin{equation}\label{eq:lem7c} 
f(\beta, \lambda\chi) = \frac{v(\beta; 1)}{2\pi i}\int_{1/2-iT_0}^{1/2+iT_0} F(s, \lambda\chi)x^{s-1} \, ds + O(yL^{-B-A}). \end{equation}
When $\beta = 0$, we can evaluate the left side of \eqref{eq:lem7c} directly by means of hypothesis (A0). Thus,
\begin{equation}\label{eq:lem7d} 
\frac 1{2\pi i} \int_{1/2-iT_0}^{1/2+iT_0} F(s, \lambda\chi)x^{s-1} \, ds = \rho_\chi + O(L^{-B-A}). 
\end{equation}
The desired inequality \eqref{eq:lem7a} follows from \eqref{eq:lem7c} and \eqref{eq:lem7d}.
\end{proof}

\begin{lemma} \label{lemma KLi}
Let $x^{7/12+2\epsilon} \le y \le x^{1-\epsilon}$ and suppose that $P,Q$ satisfy
\[ PQ \le yx^{k-1}, \quad Q \ge x^{k-5/12+\epsilon}. \]
Suppose also that $\nu > 1$. Then, for any given $A>0$,
\begin{equation}\label{eq:lemV.0}
\sum_{r \leq P}r^{-\nu} \sideset{}{^*}\sum_{\chi \bmod r} \max_{|\beta|\leq 1/(rQ)} |f(\beta, \mathbf 1_{\mathbb P}\chi) - \delta_\chi L^{-1}v(\beta;1)| \ll yL^{-A}.
\end{equation}
\end{lemma}

\begin{proof}
This is a slight variation of \cite[Lemma 4.7]{kumchev2012li}. We use the same argument, but we alter slightly the choice of $T$ in \cite[p. 620]{kumchev2012li}: instead of $T = (x/y)^2x^{3\epsilon}$, we choose 
\[ T = x^{\epsilon}\max \big( xy^{-1}, x^kQ^{-1} \big), \]
which suffices to complete the proof.
\end{proof}

\subsection{The asymptotic formula for $R_{k,s}(n, \lambda; \mathfrak{M})$}

We have
\begin{equation}\label{eq5.16}
R_{k,s}(n,\lambda;\mathfrak{M}) = \sum_{p_1, \dots, p_t\in \mathcal{I}} \int_{\mathfrak{M}} f(\alpha,\lambda)f(\alpha,\lambda^{+})^{4}e(-n_{\mathbf{p}}\alpha) \, d\alpha,
\end{equation}
where $t=s-5$ and $n_{\mathbf{p}}=n-p_1^k-\dots-p_t^k$. We now proceed to show that, for any fixed $A > 0$, one has 
\begin{equation}\label{error term after replacement} 
\int_{\mathfrak{M}} \big(f(\alpha,\lambda)f(\alpha,\lambda^{+})^{4} - f^*(\alpha,\lambda) f^*(\alpha,\lambda^{+})^{4} \big) e(-n_{\mathbf{p}}\alpha) \, d\alpha \ll y^4x^{1-k}L^{-A}.
\end{equation}

Let $\alpha\in\mathfrak{M}(q,a)$ and write $\beta = \alpha - a/q$. Since $q \le P$, property (A1) ensures that the function $\lambda$ is supported on integers $m$ with $(m,q) = 1$. Hence, by the orthogonality of the characters modulo $q$, we have
\begin{align*}
f(\alpha,\lambda) &= \sum_{\substack{1 \le h \le q\\(h,q) = 1}} e(ah^k/q) \sum_{\substack{ m \in \mathcal I\\ m \equiv h \!\!\!\! \pmod{q}}} \lambda(m)e(m^k\beta)\\
&= \phi(q)^{-1}\sum_{\chi \bmod q}S(\chi,a)f(\beta, \lambda\chi),
\end{align*}
where 
\[ S(\chi, a)=\sum_{h=1}^{q}\bar{\chi}(h)e(ah^k/q). \]
Hence, 
\begin{equation} \label{f f star Delta}
f(\alpha,\lambda)=f^*(\alpha,\lambda)+\Delta(\alpha,\lambda),
\end{equation}
where 
\begin{gather*}
\Delta(\alpha,\lambda)=\phi(q)^{-1}\sum_{\chi \bmod q} S(\chi,a)W(\beta, \lambda\chi), \\
W(\beta, \lambda\chi)=f(\beta, \lambda\chi-\rho_{\chi}), \quad \rho_{\chi}=\delta_\chi \kappa L^{-1}.
\end{gather*}
Using \eqref{f f star Delta}, we can express the integral in \eqref{error term after replacement} as the linear combination of integrals of the form
\begin{equation}\label{general form}
\int_{\mathfrak{M}}f^*(\alpha,\lambda)^a\Delta(\alpha,\lambda)^{1-a} f^*(\alpha,\lambda^{+})^b \Delta(\alpha,\lambda^{+})^{4-b} e(-n_{\mathbf{p}}\alpha) \, d\alpha,
\end{equation}
where $a\in \{0,1\}$, $b\in \{0, 1, \cdots, 4\}$ and $a+b<5$. The estimation of all those integrals follows the same pattern, so we shall focus on the most troublesome among them, namely,
\begin{equation}\label{troublesome term}
\int_{\mathfrak{M}}\Delta(\alpha,\lambda) \Delta(\alpha,\lambda^{+})^{4}e(-n_{\mathbf{p}}\alpha) \, d\alpha.
\end{equation}

We can rewrite \eqref{troublesome term} as the multiple sum
\begin{equation}\label{B J}
\sum_{q\leq P} \sum_{\chi_1 \bmod q}\cdots\sum_{\chi_5\bmod q} B(q; \chi_1, \dots, \chi_5)J(q; \chi_1, \dots, \chi_5),
\end{equation}
where
\begin{gather*}
B(q; \chi_1, \dots, \chi_5)=\phi(q)^{-5}\sum_{\substack{1\leq a\leq q\\ (a,q)=1}} S(\chi_1, a)\cdots S(\chi_5, a)e(-an_{\mathbf{p}}/q), \\
J(q; \chi_1, \dots, \chi_s)=\int_{-1/qQ}^{1/qQ}W(\beta, \lambda\chi_1)W(\beta, \lambda^{+}\chi_2)\cdots W(\beta, \lambda^{+}\chi_5) e(-n_{\mathbf{p}}\beta) \, d\beta.
\end{gather*}
First, we reduce \eqref{B J} to a sum over primitive characters. If $\chi$ is a Dirichlet character modulo $q$ that is induced by a primitive character $\chi^*$ modulo $r$, $r \mid q$, then by property (A1), $\lambda^\pm\chi = \lambda^\pm\chi^*$. Thus, 
\begin{equation} \label{reduced W}
W(\beta, \lambda^\pm\chi)=W(\beta, \lambda^\pm\chi^*).
\end{equation}
Let $\chi_{i}^{*}$ modulo $r_i$, $r_i|q$, be the primitive character inducing $\chi_i$ and set $q_0 = [r_1, \dots, r_5]$. By \eqref{reduced W}, we have
\begin{equation*}
J(q; \chi_1, \dots, \chi_5) = J(q; \chi_1^*, \dots, \chi_5^*).
\end{equation*}
Therefore, the sum \eqref{B J} does not exceed
\[ \sum_{r_1 \le P} \; \sideset{}{^*}\sum_{\chi_1 \bmod r_1} \cdots \sum_{r_5 \le P} \; \sideset{}{^*}\sum_{\chi_5 \bmod r_5} J_0(\chi_1, \dots, \chi_5) B_0(\chi_1, \dots, \chi_5), \]
where  
\begin{gather*} 
B_0(\chi_1, \dots, \chi_5) = \sum_{\substack{q\leq P\\ q_0|q}} |B(q; \chi_1, \dots, \chi_5)|, \\ 
J_0(\chi_1, \dots, \chi_5) = \int_{-1/(q_0Q)}^{1/(q_0Q)}|W(\beta, \lambda\chi_1)W(\beta, \lambda^{+}\chi_2) \cdots W(\beta, \lambda^{+}\chi_5)| \, d\beta.
\end{gather*}
Recalling the bound (see \cite[Lemma 6.1]{wei2014sums})
\[ B_0(\chi_1, \dots, \chi_5)\ll q_0^{-3/2+\epsilon}L^c, \]
we conclude that the sum \eqref{B J} is
\begin{equation}\label{eq5.21}
\ll L^c \sum_{r_1 \le P} \; \sideset{}{^*}\sum_{\chi_1 \bmod r_1} \cdots \sum_{r_5 \le P} \; \sideset{}{^*}\sum_{\chi_5 \bmod r_5} q_0^{-3/2+\epsilon} V(\lambda\chi_1)  V(\lambda^{+}\chi_{2}) V(\lambda^{+}\chi_{3}) W(\lambda^{+}\chi_{4}) W(\lambda^{+}\chi_{5}), 
\end{equation}
where for a character $\chi$ modulo $r$, we write 
\begin{gather*}
V(\lambda\chi)=\max_{|\beta|\leq 1/(rQ)}|W(\beta, \lambda\chi)|, \\
W(\lambda\chi)=\bigg(\int_{-1/(rQ)}^{1/(rQ)}|W(\beta, \lambda\chi)|^2 \, d\beta \bigg)^{1/2}.
\end{gather*}

Next, we proceed to estimate the sum in \eqref{eq5.21} by Lemmas \ref{W lemma}, \ref{V lemma} and \ref{main lemma}, which we will denote by $\Sigma$. When $y = x^{\theta}$ with $\theta > 31/40$ and $\delta \le 2(\theta - 31/40)$, the definitions of $P$ and $Q$ (recall \eqref{P Q}) ensure that they satisfy inequalities \eqref{eq5.9}. Since the sieve functions $\lambda^\pm$ have properties (A0)--(A3), this means that all the hypotheses of the lemmas are in place.

To begin the estimation of $\Sigma$, we note that Lemma \ref{W lemma} yields 
\begin{equation}\label{eq5.23}
\sum_{r \le P} \; \sideset{}{^*}\sum_{\chi \!\!\!\! \mod r} [g, r]^{-\nu}W(\lambda^{+}\chi) \ll g^{-\nu+\epsilon} y^{1/2} x^{(1-k)/2} L^c + g^{-\nu}I_0^{1/2},
\end{equation} 
where  
\begin{equation}\label{I-bound}
\begin{split}
I_0 =\int_{-1/Q}^{1/Q}|v(\beta; 1)|^2 \, d\beta &\ll \iint_{\mathcal I^2} \frac{du_1 du_2}{Q+|u_1^k-u_2^k|} \\
&\ll  yx^{1-k}+yLQ^{-1} \ll yx^{1-k}.
\end{split}
\end{equation}
(We remark that the second term on the right side of \eqref{eq5.23} accounts for the contribution of $\rho_\chi$ to $W(\beta, \lambda\chi)$---which is present only when $r = 1$.) Similarly, the first part of Lemma \ref{V lemma} yields
\begin{equation}\label{eq5.24}
\sum_{r \le P} \; \sideset{}{^*}\sum_{\chi \!\!\!\! \mod r} [g, r]^{-\nu}V(\lambda^{+}\chi) \ll g^{-\nu+\epsilon} yL^c.
\end{equation} 
Applying \eqref{eq5.23} to the summations over $r_5$ and $r_4$ in $\Sigma$ and then \eqref{eq5.24} to the summations over $r_3$ and $r_2$, we obtain
\[ \Sigma \ll y^3x^{1-k}L^c \sum_{r \le P} \; \sideset{}{^*}\sum_{\chi \bmod r} r^{-3/2+5\epsilon} V(\lambda\chi). \]
Finally, we apply Lemma \ref{main lemma} to the last sum and conclude that
\[ \Sigma \ll y^4x^{1-k}L^{-A} \] 
for any fixed $A>0$. This inequality and its variants for other integrals of the form \eqref{general form} establish \eqref{error term after replacement}.

Having established \eqref{error term after replacement}, we can combine it with \eqref{eq5.16} to get
\[ R_{k,s}(n, \lambda; \mathfrak M) = \int_{\mathfrak{M}} f(\alpha,\mathbf 1_{\mathbb P})^t f^*(\alpha, \lambda) f^*(\alpha, \lambda^+)^4 e(-n\alpha) \, d\alpha + O\big( y^{s-1}x^{1-k}L^{-A} \big). \]
We now define a new, slimmer set of major arcs $\mathfrak M_0$, given by \eqref{major and minor} with $Q_0 = x^{k-1}yP^{-1}$ in place of $Q$. From the bound
\[ f^*(\alpha, \lambda^\pm) \ll yq^{-1/2+\epsilon}\big( 1 + yx^{k-1}|\alpha - a/q| \big)^{-1/2} \quad \text{if } \alpha \in \mathfrak{M}(q,a), \]
we find that
\begin{align*} 
\int_{\mathfrak M \setminus \mathfrak{M}_0} \big| f(\alpha,\mathbf 1_{\mathbb P})^t f^*(\alpha, \lambda) f^*(\alpha, \lambda^+)^4 \big| \, d\alpha 
&\ll \sum_{\substack{1 \le a \le q \le P\\(a, q) = 1}} \int_{|\beta| \ge 1/(qQ_0)} \frac {y^sq^{-5/2+\epsilon}}{(1 + yx^{k-1}|\beta|)^{5/2}} \, d\beta \\
&\ll y^{s-1}x^{1-k}P^{-1/2+\epsilon}.
\end{align*}
Hence, for any fixed $A > 0$, we have
\begin{equation}\label{eq5.26} 
R_{k,s}(n, \lambda; \mathfrak M) = \int_{\mathfrak{M}_0} f(\alpha, \mathbf 1_{\mathbb P})^t f^*(\alpha, \lambda) f^*(\alpha, \lambda^+)^4 e(-n\alpha) \, d\alpha + O\big( y^{s-1}x^{1-k}L^{-A} \big). 
\end{equation} 

Finally, we have
\begin{equation}\label{eq5.27} 
\int_{\mathfrak{M}_0} \big( f(\alpha, \mathbf 1_{\mathbb P})^t - f^*(\alpha)^t \big) f^*(\alpha, \lambda) f^*(\alpha, \lambda^+)^4 e(-n\alpha) \, d\alpha \ll y^{s-1}x^{1-k}L^{-A}. 
\end{equation} 
The proof of this inequality is simlar to the proof of \eqref{error term after replacement}, except that we do not need to use Lemma \ref{W lemma} (the bound \eqref{I-bound} can be used instead) and we use Lemma \ref{lemma KLi} instead of Lemma~\ref{main lemma}. We remark that during the process, we need to verify the hypotheses $Q \ge x^{k-9/20}$ and $Q \ge x^{k-5/12+\epsilon}$ of those lemmas for $Q = Q_0$; with our choice of $Q_0$, those hypotheses are satisfied when $y \ge x^{7/12 + \delta}$. 

By \eqref{eq5.26} and \eqref{eq5.27}, we have
\[ R_{k,s}(n, \lambda; \mathfrak M) = \kappa\kappa_+^4 \int_{\mathfrak{M}_0} f^*(\alpha)^s e(-n\alpha) \, d\alpha + O\big( y^{s-1}x^{1-k}L^{-A} \big). \]
The evaluation of the last integral uses standard major arc techniques (e.g., see Wei and Wooley \cite[pp. 1150--1151]{wei2014sums}), so we can omit it and report that
\[ \int_{\mathfrak{M}_0} f^*(\alpha)^s e(-n\alpha) \, d\alpha = \mathfrak S(n)\mathfrak I(n)L^{-s} + O\big( y^{s-1}x^{1-k}P^{-1} \big). \]
We note that $\mathfrak S(n)$ is the standard singular series in the Waring--Goldbach problem for $s$ $k$th powers. In particular, it is known that $1 \ll \mathfrak S(n) \ll 1$ when $n \in \mathcal{H}_{k,s}$. Since the inequality
\[ y^{s-1}x^{1-k} \ll \mathfrak{I}(n) \ll y^{s-1}x^{1-k} \]
is also standard (compare to \cite[(6.5)]{wei2014sums}), we conclude that \eqref{major arc bound} holds with 
\[ \mathfrak C(n) = \mathfrak S(n)\mathfrak I(n)y^{1-s}x^{k-1}. \]

{\bf Acknowledgments.}
The second author would like to thank Professor Jianya Liu for his constant encouragement. He also wants to thank the China Scholarship Council (CSC) for supporting his studies in the United States and the Department of Mathematics at Towson University for the hospitality and the excellent conditions.


\begin{thebibliography}{20}

\bibitem{bakerharman}
R. C. Baker, G. Harman, and J. Pintz.
\emph{The exceptional set for Goldbach's problem in short intervals}. 
In ``Sieve Methods, Exponential Sums, and Their Applications in Number Theory'', pp. 1--54. 
Cambridge University Press, 1997. 

\bibitem{bourgain2015VMT}
J. Bourgain, C. Demeter, and L. Guth.
\emph{Proof of the main conjecture in Vinogradov's mean value theorem for degrees higher than three}. 
Ann. of Math. (2) \textbf{184} (2016), 633--682. 

\bibitem{bruedernfouvry}
J. Br\"udern and E. Fouvry.
\emph{Lagrange's four squares theorem with almost prime variables}.
J. reine angew. Math. \textbf{454} (1994), 59--96.

\bibitem{choi2006kumchev}   
S.~K.~K. Choi and A.~V. Kumchev. 
\emph{Mean values of Dirichlet polynomials and applications to linear equations with prime variables}.
Acta Arith. \textbf{123} (2006), 125--142. 

\bibitem{daemen2010asymptotic}
D.~Daemen. 
\emph{The asymptotic formula for localized solutions in Waring's problem and approximations to Weyl sums}. 
Bull. London Math. Soc. \textbf{42} (2010), 75--82. 


\bibitem{hua1938}
L.~K. Hua. 
\emph{Some results in additive prime number theory}. 
Quart. J. Math. (Oxford) \textbf{9} (1938), 68--80.

\bibitem{hua1965}
L.~K. Hua. 
\emph{Additive Theory of Prime Numbers}. 
American Mathematical Society, 1965.

\bibitem{huang2015exponential}
B.~R. Huang. 
\emph{Exponential sums over primes in short intervals and an application to the Waring--Goldbach problem}. 
Mathematika \textbf{62} (2016), 508--523.

\bibitem{huxley1972difference}
M.~N. Huxley. 
\emph{On the difference between consecutive primes}. 
Invent. Math. \textbf{15} (1972), 164--170.

\bibitem{kumchev2012li}     
A.~V. Kumchev and T.~Y. Li. 
\emph{Sums of almost equal squares of primes}. 
J. Number Theory \textbf{132} (2012), 608--636. 

\bibitem{liuzhan1996}
J.~Y. Liu and T.~Zhan. 
\emph{On sums of five almost equal prime squares}.
Acta Arith. \textbf{77} (1996), 369--383.

\bibitem{montgomery book}
H.~L. Montgomery and R.~C. Vaughan. 
\emph{Multiplicative Number Theory: I. Classical Theory}. 
Cambridge University Press, 2007.

\bibitem{titchmarsh book}
E.~C. Titchmarsh. 
\emph{The Theory of the Riemann Zeta-Function}, 2nd ed., revised by D. R. Heath-Brown. 
Oxford University Press, 1986.

\bibitem{vaughan book}
R.~C. Vaughan. 
\emph{The Hardy--Littlewood Method}, 2nd ed. 
Cambridge University Press, 1997.

\bibitem{wei2014sums}
B.~Wei and T.~D. Wooley. 
\emph{On sums of powers of almost equal primes}.
Proc. London Math. Soc. (3) \textbf{111} (2015), 1130--1162.

\bibitem{wooley2016cubic}
T.~D. Wooley. 
\emph{The cubic case of the main conjecture in Vinogradov’s mean value theorem}. 
Adv. Math. \textbf{294} (2016), 532--561.
\end{thebibliography}
\end{document}